\def\timestamp{%
Time-stamp: <PFA-free.tex: Monday 25-10-2021 at 17:48:05 (cest)>}
\def\stripname Time-stamp: <#1 #2>{#2}
\edef\filedate{\expandafter\stripname\timestamp}

\documentclass[a4paper]{amsart}

\newcommand\Fn{\operatorname{Fn}}
\newcommand\cl{\operatorname{cl}}

\newcommand\cee{\mathfrak{c}}

\newcommand\axiom{\mathsf}
\newcommand\PFA{\axiom{PFA}}
\newcommand\MA{\axiom{MA}}
\newcommand\CH{\axiom{CH}}

\DeclareMathSymbol\PP0{AMSb}{`P}
\DeclareMathSymbol\Q0{AMSb}{`Q}
\DeclareMathSymbol\T0{AMSb}{`T}
\DeclareMathSymbol\le    \mathrel{AMSa}{"36}
\DeclareMathSymbol\ge    \mathrel{AMSa}{"3E}
\DeclareMathSymbol\forces\mathrel{AMSa}{"0D}

\newcommand\orpr[2]{\langle{#1},{#2}\rangle}

\newcommand\omegaoneseq[2][\omega_1]{\langle{#2}_\alpha:\alpha\in#1\rangle}

\newcommand\calA{\mathcal{A}}
\newcommand\calU{\mathcal{U}}
\newcommand\pow{\mathcal{P}}

\theoremstyle{plain}
\newtheorem{theorem}{Theorem}
\newtheorem{lemma}[theorem]{Lemma}

\theoremstyle{definition}
\newtheorem{definition}[theorem]{Definition}

\usepackage{amsrefs}

\begin{document}

\title{$\PFA$ and $\omega_1$-free compact spaces}

\author[A. Dow]{Alan Dow}
\address{
Department of Mathematics\\
UNC-Charlotte\\
9201 University City Blvd. \\
Charlotte, NC 28223-0001}
\email{adow@uncc.edu}
\urladdr{https://webpages.uncc.edu/adow}

\author[K. P. Hart]
       {Klaas Pieter Hart}

\address{
Faculty EEMCS\\
TU Delft\\
Postbus 5031\\
2600~GA {} Delft\\
the Netherlands}
\email{k.p.hart@tudelft.nl}
\urladdr{http://fa.its.tudelft.nl/\~{}hart}



\date{\filedate}

\begin{abstract}
The Proper Forcing Axiom implies that compact Hausdorff spaces are either 
first-countable or contain a converging $\omega_1$-sequence.
\end{abstract}

\maketitle

\section*{Introduction}

This paper continues the investigation from~\cite{MR3194414} into
a conjectured dichotomy for compact Hausdorff spaces: every such space
is either first-countable or it has a non-trivial converging 
$\omega_1$-sequence.
Whether this dichotomy holds was asked by Juh\'asz and a consistent 
counterexample to this conjecture was given by him,
Koszmider, and Soukup in~\cite{MR2519221}.

The main result in~\cite{MR3194414} states that the dichotomy holds in a 
large class of ccc forcing extensions of models in which the Continuum 
Hypothesis holds. 
This result does not encompass all ccc extensions and one can rightfully ask
what the status of the dichotomy is under Martin's Axiom.

In Section~\ref{sec:MA} we point out that a space constructed
by the first author in~\cite{MR3426909} witnesses that 
the combination $\MA+{\cee=\aleph_2}$ is not strong enough to
imply the dichotomy.

Our main result appears in Section~\ref{sec:PFA}: the Proper Forcing Axiom
implies the dichotomy.
This then also implies that the dichotomy is consistent with and independent
of $\MA+{\cee=\aleph_2}$.
It also improves upon the first author's result 
in~\cite{MR3205487}*{Theorem~5.1},
which implies that, in turn, $\PFA$~implies that a compact space without
convergent $\omega_1$-sequences must be Fr\'echet-Urysohn.

\section{Some preliminaries}

Here we collect the notions that will be used throughout the paper.
Others will be defined when needed.

\subsection*{Converging $\omega_1$-sequences}

The central notion is that of a converging $\omega_1$-se\-quence. 
An $\omega_1$-sequence $\omegaoneseq{x}$ \emph{converges} to a point~$x$
if for every neighbourhood~$U$ of~$x$ there is an~$\alpha$ such that
$\{x_\beta:\beta\ge\alpha\}\subseteq U$.
Unless stated otherwise our converging sequences are assumed to be non-trivial,
that is, injective.
A space that contains no non-trivial converging $\omega_1$-sequences
will be called \emph{$\omega_1$-free}.

In all cases where we construct converging $\omega_1$-sequences we use
the regularity of the space to show convergence:
an $\omega_1$-sequence $\omegaoneseq{x}$ converges to a point~$x$ if and only
if for every neighbourhood~$U$ of~$x$ there is an~$\alpha$ such that
$\{x_\beta:\beta\ge\alpha\}\subseteq\cl U$.

We should record here that the dichotomy is indeed a dichotomy.
If a point has a countable local base then it is not the limit
of a non-trivial $\omega_1$-sequence: that sequence would have to be constant
on a tail.

\subsection*{Free sequences}

At one point we shall need to know that a compact $\omega_1$-free
space has countable tightness, which for the purposes of this
paper is best defined by a characterization:
there is no free $\omega_1$-sequence.

An $\omega_1$-sequence $\omegaoneseq{x}$ is \emph{free} if for every~$\alpha$
the sets $\{x_\beta:\beta<\alpha\}$ and $\{x_\beta:\beta\ge\alpha\}$ have 
disjoint closures.

\section{$\MA$ is not enough}
\label{sec:MA}

In this section we show that the conjuction of $\MA$ and
the equality $\cee=\aleph_2$ is not strong enough to give a positive
answer to Juh\'asz' question.
For this we recall the definition of initial $\omega_1$-compactness:
every open cover of cardinality at most~$\aleph_1$ has a finite subcover.

In \cite{MR3426909}*{Corollary~5.5} one finds a 
model of $\MA+{\cee=\aleph_2}$ in which there is a compact space~$X$
that is of countable tightness, and has a proper dense subspace~$Y$
that is first-countable and initially $\omega_1$-compact.
As $Y$ is a fortiori countably compact the space~$X$ is not first-countable
at all points that are not in~$Y$.
In addition every compact subset of $X\setminus Y$ is finite.
We shall show that $X$ is $\omega_1$-free.

Assume $\omegaoneseq{x}$ is a converging $\omega_1$-sequence with limit~$x$.
Since, as noted above, a point of first-countability cannot be the 
limit of a non-trivial $\omega_1$-sequence we must have $x\in X\setminus Y$.

Since $x$~is the only complete accumulation point of the set
$\{x_\alpha:\alpha\in\omega_1\}$ the initial $\omega_1$-compactness of~$Y$
implies that all but countably many~$x_\alpha$ are in~$X\setminus Y$.

Thus we assume the~$x_\alpha$ are all in~$X\setminus Y$.
For each limit ordinal~$\delta$ we choose a cofinal subset~$C_\delta$
of order type~$\omega$ and an accumulation point~$y_\delta$ of
$\{x_\alpha:\alpha\in C_\delta\}$; by the assumption on~$X\setminus Y$
we can take $y_\delta$ in~$Y$. 

There is a point~$z$ in~$Y$ with the property that for every 
neighbourhood~$U$ of~$z$ the set $\{\delta:y_\delta\in U\}$ is uncountable.
But then so is $\{\alpha:x_\alpha\in U\}$ for every such neighbourhood.
But then $z=x$ and we have our contradiction.

\section{The main result}
\label{sec:PFA}

In this section we prove our main result.

\begin{theorem}[$\PFA$]
Every compact Hausdorff space either contains a non-trivial converging
$\omega_1$-sequence or is first-countable.
\end{theorem}

We shall prove the theorem in a number of steps.
We assume $\PFA$ and we let $X$ be a compact Hausdorff space that
is non first-countable and also $\omega_1$-free, and we intend to reach
a contradiction by the end of this section.

\begin{lemma}
Our space $X$ has countable tightness.  
\end{lemma}

\begin{proof}
This is established by Juh\'asz and Szentmikl\'ossy in~\cite{MR1137223}:  
a compact space of uncountable tightness contains (even) a \emph{free} 
$\omega_1$-sequence that converges.
\end{proof}

\begin{lemma}
We may assume, without loss of generality, that $X$ is separable.  
\end{lemma}

\begin{proof}
This was established by the present authors in~\cite{MR3194414}:
a compact $\omega_1$-free space that is not first-countable contains
a closed separable subspace that is not first-countable.  
\end{proof}

Henceforth we add separability to the assumptions on our space~$X$.

\begin{lemma}\label{lemma:card-cee}
Our space $X$ has cardinality at most~$\cee$.  
\end{lemma}

\begin{proof}
In~\cite{MR930252} Balogh proved that $\PFA$ implies all compact
Hausdorff spaces of countable tightness are in fact sequential.

Now apply the elementary fact that the sequential closure of a countable
set has cardinality at most~$\cee$.
\end{proof}

Our next step is to investigate what happens to our space~$X$ when
we force with the partial order~$\Fn(\omega_1,2,\aleph_1)$: 
the set of countable partial functions from~$\omega_1$ to~$2$
(notation as in~\cite{MR597342}).

The reason for this is that the main steps towards the contradiction
involve proper partial orders that are built in two stages,
the first of which is the aforementioned partial 
order~$\Fn(\omega_1,2,\aleph_1)$;
this partial order is proper and forces $\CH$ to hold.

\begin{lemma}
Upon forcing with $\Fn(\omega_1,2,\aleph_1)$
our space~$X$ remains compact.  
\end{lemma}

\begin{proof}
By ``our space~$X$ remains compact'' we mean that the set $X$, with the
topology generated by the ground-model open sets, is a compact space.

To begin: our space~$X$ does remain countably compact.
This is so because every countably infinite subset of~$X$ belongs to the
ground model and hence has an accumulation point there.
Because the old sets form a base for the new topology that point remains
an accumulation point.

Now, if $X$ is no longer compact in this extension then 
\cite{MR1031969}*{Proposition~5.8} applies and there is a proper partial
order~$\PP$ that introduces an embedding~$e$ of the ordinal space~$\omega_1$
into~$X$.
We shall specify $\aleph_1$ many dense sets in the partial order
$\Fn(\omega_1,2,\aleph_1)*\dot\PP$ to determine a map
$f:\omega_1\to X$, in our ground model, that is a free sequence in~$X$.

First we take, for each~$\alpha$, the set 
$E_\alpha=\{p:(\exists x\in X)(p\forces \dot e(\alpha)=x)\}$.
An application of $\PFA$ to this family yields a map $f:\omega_1\to X$.

Next we note that in the extension the set $e[\alpha+1]$ is compact and 
does not meet the closure of~$e\bigl[[\alpha+1,\omega_1)\bigr]$.
Therefore there is an open set~$U_\alpha$ in the ground model topology
with the property that $e[\alpha+1]\subseteq U_\alpha$ and
$\cl U_\alpha\cap e\bigl[[\alpha+1,\omega_1)\bigr]=\emptyset$.
As above we get dense sets~$D_\alpha$ that determine the 
sequence~$\omegaoneseq{U}$.

Now apply $\PFA$ to the union of the families to get the desired
free sequence $f:\omega_1\to X$.

This contradicts the countable tightness of~$X$ and establishes the lemma.
\end{proof}

If we combine this with Lemma~\ref{lemma:card-cee} then we know that in
the extension by $\Fn(\omega_1,2,\aleph_1)$ our space~$X$ has become a compact
space of cardinality and weight at most~$\aleph_1$.
This will be used in the next step.

From~\cite{MR3194414} we quote (and modify) the notion ultra-Fr\'echet, 
a useful technical property that will be used in our proof.

\begin{definition}\label{def.u-F}
A space $Z$ will be said to be \emph{ultra-Fr\'echet} if it has countable 
tightness and for each countable subset~$D$ of~$Z$ and each free 
ultrafilter $\calU$ on~$D$ that converges in~$Z$ there is a countable 
subfamily~$\calU'$ of~$\calU$ with the property that every infinite 
pseudointersection of~$\calU'$ converges.
\end{definition}

This is a weakening of the original definition, which did not require
that the ultrafilter~$\calU$ be convergent and thus would imply that
the space is countably compact.
In the present context of compact spaces there is no difference but the
notion may be useful outside this class.

We also remark that all pseudointersections of the countable subfamily
converge to the same point (the union of two pseudointersections is again
a pseudointersection), to wit the limit of the ultrafilter.

\begin{lemma}
Our space~$X$ is ultra-Fr\'echet, also upon forcing 
with~$\Fn(\omega_1,2,\aleph_1)$.  
\end{lemma}

\begin{proof}
We argue by contradiction and assume that in the extension there are a 
countable set~$D$ and an ultrafilter~$\calU$ on~$D$ where the property fails; 
that is: every countable subfamily of~$\calU$ has an infinite 
pseudointersection that does not converge.

This will also cover the ground model case: because the forcing adds no new 
countable sets a `bad' pair $\orpr D\calU$ from the ground model is still
`bad' in the extension.

We note that by sequentiality of~$X$ every countably infinite subset of~$X$
contains an infinite subset that converges: if the set is closed then
it is compact metrizable and hence contains a convergent sequence, and
if it is not closed then sequentiality gives us a sequence in the set
that converges to a point outside the set.

We now construct, recursively, an almost disjoint family 
$\{a_\alpha:\alpha\in\omega_1\}$ of infinite subsets of~$D$ that converge,
say $a_\alpha$~converges to~$y_\alpha$, where $y_\alpha\neq x$.
Note that none of the~$a_\alpha$s can belong to~$\calU$: the countable,
indeed singleton, family~$\{a_\alpha\}$ would contradict our assumption.
We also choose, for each~$\alpha$, a neighbourhood~$O_\alpha$ of~$x$ such 
that~$y_\alpha\notin\cl O_\alpha$.

We enumerate $\calU$ as $\omegaoneseq{U}$.

At stage $\alpha$ we let $c_\alpha$ be an infinite pseudointersection
of $\{U_\beta:\beta<\alpha\}\cup\{D\cap O_\beta:\beta<\alpha\}$
that does not converge.  

As noted above $c_\alpha$ has an infinite subset that converges; 
if every infinite convergent subset of~$c_\alpha$ were to converge to~$x$
then by sequentiality the sequential closure of~$c_\alpha$ would be 
just~$c_\alpha\cup\{x\}$, and $c_\alpha$ would converge after all.
So there is an infinite convergent subset~$a_\alpha$ of~$c_\alpha$ that
converges to some point~$y_\alpha$ distinct from~$x$.

The family $\{a_\alpha:\alpha\in\omega_1\}$ and the ultrafilter~$\calU$ satisfy
the conditions of~\cite{MR2897749}*{Lemma~1.8}.
Therefore there is a ccc partial order~$\Q$ that introduces an uncountable
subset~$I$ of~$\omega_1$ such that $\{a_\alpha:\alpha\in I\}$
is a Luzin family.
This means that for every $\alpha\in I$ and every natural number~$n$
the set $F(\alpha,n)=\{\beta\in I\cap\alpha:a_\beta\cap a_\alpha\subseteq n\}$ 
is finite.

It takes only $\aleph_1$ many dense sets in the partial order
$\Fn(\omega_1,2,\aleph_1)*\dot\Q$ to determine such an uncountable
set~$I$, the function $\alpha\mapsto a_\alpha\cup\{y_\alpha\}$ and the finite 
sets~$F(\alpha,n)$.
Therefore an application of~$\PFA$ yields the existence of such sets
in the ground model, and we may as well assume $I=\omega_1$.

The sequence $\omegaoneseq{y}$ converges, in the ground model.
Indeed, let $K=\bigcap_\alpha\cl\{y_\beta:\beta\ge\alpha\}$.
By compactness $K$~is non-empty of course and by the Luzin property
of the family $\{a_\alpha:\alpha\in\omega_1\}$ it consists of just one point
and $\omegaoneseq{y}$ converges to that point, which contradicts
our standing assumption on~$X$.

To see this, let $x\in K$ and let $O$ be an arbitrary neighbourhood of~$x$.
The set $A=\{\alpha:y_\alpha\in O\}$ is uncountable, and for each
$\alpha\in A$ there is an~$m_\alpha\in\omega$ such 
that $a_\alpha\setminus O\subseteq m_\alpha$.
We may assume that there is a single~$m$ such that $m_\alpha=m$ 
for all~$\alpha\in A$.

We claim that $\{\beta:y_\beta\notin\cl O\}$ is countable; by regularity of~$X$
this suffices to prove that $\omegaoneseq{y}$ converges to~$x$.

If the set is uncountable then, by the same argument as above, 
we find an uncountable set~$B$ and 
an~$n\ge m$ such that $a_\beta\cap O\subseteq n$ for all~$\beta\in B$.
But now take $\alpha\in A$ such that $\alpha\cap B$ is infinite.
Then $a_\beta\cap a_\alpha\subseteq n$ for all $\beta\in \alpha\cap B$,
contradicting the Luzin property of our family.
\end{proof}

We can now take the last step towards our contradiction.
To recapitulate: 
we assume our space~$X$ is compact, separable, not first-countable 
and~$\omega_1$-free.
From this we deduced that $X$~remains compact after forcing with
the partial order~$\Fn(\omega_1,2,\aleph_1)$, and that it is
ultra-Fr\'echet before and after this forcing.
In addition, its separability and non-first-countability are preserved
by~$\Fn(\omega_1,2,\aleph_1)$ too.

\begin{lemma}
The assumptions that $X$ is compact, separable,
not first-countable and ultra-Fr\'echet,
before and after forcing with~$\Fn(\omega_1,2,\aleph_1)$, imply that
$X$~has a converging $\omega_1$-sequence.  
\end{lemma}
 
\begin{proof}
We work in the extension by~$\Fn(\omega_1,2,\aleph_1)$ and take a
sequence~$\omegaoneseq{M}$ of countable elementary substructures
of~$H(\kappa)$ for a suitable large cardinal~$\kappa$ such that 
\begin{itemize}
\item $X$ belongs to~$M_0$ 
\item $M_\alpha=\bigcup_{\beta<\alpha}M_\beta$ if $\alpha$ is a limit
\item $\langle M_\beta:\beta\le\alpha\rangle\in M_{\alpha+1}$ 
      and $M_\alpha\prec M_{\alpha+1}$ for
      all~$\alpha$
\end{itemize}
By elementarity there are $z\in M_0\cap X$ and $D\in M_0$ such that $X$~is 
not first-countable at~$z$ and $D$~is countable and dense in~$X$.
We fix such~$z$ and~$D$.

Elementarity also implies that $M_\alpha^\omega\in M_{\alpha+1}$ for all~$\alpha$
and this, combined with the Continuum Hypothesis, implies that every
countable subset of~$M=\bigcup_{\alpha<\omega_1}M_\alpha$ is a member of~$M$.
In particular $\pow(D)\subseteq M$.

At this point we could refer to~\cite{MR3194414}*{Proof of Proposition~2.8}, 
which produces sequences~$\omegaoneseq{a}$ and~$\omegaoneseq{b}$ of infinite 
subsets of~$D$ and a sequence~$\omegaoneseq{y}$ of points 
in~$X\setminus\{z\}$ with the following properties:
\begin{enumerate}
\item $a_\alpha$ converges to~$z$ 
\item $b_\alpha$ converges to~$y_\alpha$
\item $a_\alpha\cap b_\alpha=\emptyset$
\item $\{A:a_\alpha\cup b_\alpha\subseteq^* A\}$ is an ultrafilter
      in~$M_\alpha\cap\pow(D)$
\item $a_\alpha$, $b_\alpha$, and $y_\alpha$ belong to~$M_{\alpha+1}$
\end{enumerate}
However, for the benefit of the reader we will give a somewhat more streamlined 
construction than the one referred to above.

Fix an~$\alpha\in\omega_1$ and let $y_\alpha$ be a point in the intersection
$$
\bigcap\{U:U\in M_\alpha, z\in U, U \text{ is open in }X\}
$$
that is distinct from~$z$. 
Such a point exists because $\{z\}$~is not a $G_\delta$-set in~$X$.
Moreover, by elementarity we can assume $y_\alpha\in M_{\alpha+1}$.

Next observe that $y_\alpha\notin\cl A$ whenever $A\in M_\alpha\cap\pow(D)$
and $z\notin\cl A$.
There is an ultrafilter~$\calU_1$ on the Boolean algebra $M_\alpha\cap\pow(D)$
that extends the filter dual to $\{A:z\notin\cl A\}$ and is such that
$y_\alpha\in\bigcap\{\cl U:U\in\calU_1\}$; 
again we can assume $\calU_1\in M_{\alpha+1}$.
By the contrapositive of the first sentence of this paragraph we have
$z\in\bigcap\{\cl U:U\in\calU_1\}$.
Next we take ultrafilters $\calU_2$ and $\calU_3$ on~$D$ that extend~$\calU_1$
and converge to~$z$ and~$y_\alpha$ respectively, yet again we take these
ultrafilters in~$M_{\alpha+1}$.

Now apply the ultra-Fr\'echet property to $\calU_2$ and $y_\alpha$, and 
to~$\calU_3$ and~$z$ to find countable subfamilies $\calU_2'$ and $\calU_3'$
as in the definition, where we may assume that both families contain~$\calU_1$.
These choices can be made in~$M_{\alpha+1}$ and we can then take 
pseudointersections~$a_\alpha$ and~$b_\alpha$ of~$\calU_2'$ and~$\calU_3'$
respectively that belong to~$M_{\alpha+1}$.
By the remark after Definition~\ref{def.u-F} $a_\alpha$ converges to~$z$
and $b_\alpha$ converges to~$y_\alpha$, and we can 
assume $a_\alpha\cap b_\alpha=\emptyset$.

All five conditions are met: we took care of (1), (2), (3), and~(5)
explicitly; (4)~holds because $a_\alpha$ and $b_\alpha$ are both
pseudointersections of~$\calU_1$. 
 
Conditions (4) and (5) imply that the sequences~$\omegaoneseq{a}$ 
and~$\omegaoneseq{b}$ are \emph{not} $\sigma$-separated in a strong way:
whenever $\calA$~is a countable family of subsets of~$D$ there
is an~$\alpha$ such that for every~$A\in\calA$ either
$A\cap(a_\alpha\cup b_\alpha)$ is finite or $a_\alpha\cup b_\alpha\subseteq^*A$.

We may then apply a result of Todor\v{c}evi\'c, see~\cite{MR1711328}*{p.~145},
that implies there is a proper partial order~$\T$ that produces an uncountable
subset~$J$ of~$\omega_1$ such that
$(a_\alpha\cap b_\beta)\cup(a_\beta\cap b_\alpha)$ 
is non-empty
whenever $\alpha,\beta\in J$ are distinct.

As before we need $\aleph_1$ many dense sets in the proper partial order
$\Fn(\omega_1,2,\aleph_1)*\dot\T$ to determine the 
sequences $\omegaoneseq{a}$, $\omegaoneseq{b}$, and~$\omegaoneseq{y}$ 
as well as the uncountable set~$J$.
We apply $\PFA$ and obtain items with properties~(1), (2) and~(3)
in our list above as well as the uncountable set~$J$ from Todor\v{c}evi\'c's
result, which we may assume to be $\omega_1$ itself.

We claim that $\omegaoneseq{y}$ converges to~$z$, which is the final
contradiction that we seek.
Let $O$ be a neighbourhood of~$z$.
Because every $a_\alpha$ converges to~$z$ we know that $a_\alpha\setminus O$
is always finite.
We claim that $b_\alpha\cap O$ is finite for only countably many~$\alpha$,
and hence that there is an~$\alpha$ such that 
$y_\beta\in\cl O$ for~$\beta\ge\alpha$.

Indeed, assume $b_\alpha\cap O$ is finite for uncountably many~$\alpha$ and
fix two finite sets $F$ and~$G$ such that the set~$A$ of those~$\alpha$
for which $F=a_\alpha\setminus O$ and $G=b_\alpha\cap O$ is uncountable.
Now let $\alpha,\beta\in A$ be distinct.
Then $a_\alpha=F\cup(a_\alpha\cap O)=(a_\beta\setminus O)\cup(a_\alpha\cap O)$ 
and $b_\beta=G\cup(b_\beta\setminus O)=(b_\alpha\cap O)\cup(b_\beta\setminus O)$,
which implies $a_\alpha\cap b_\beta = \emptyset$.
Likewise $a_\beta\cap b_\alpha=\emptyset$
and we are done.
\end{proof}

\end{document}